\documentclass[abstract=on]{scrartcl}

\usepackage{typearea}
\typearea{22}
\usepackage{amsmath,amssymb}
\usepackage{amsthm}
\usepackage{mathrsfs}
\usepackage[pdftex]{graphicx}
\usepackage[colorlinks,linkcolor=blue,citecolor=blue,urlcolor=blue]{hyperref}
\newcommand{\SE}{\scriptscriptstyle\mathrm{SE}}
\newcommand{\DE}{\scriptscriptstyle\mathrm{DE}}
\newcommand{\domD}{\mathscr{D}}
\newcommand{\RR}{\mathbb{R}}
\newcommand{\CC}{\mathbb{C}}
\theoremstyle{plain}
\newtheorem{theorem}{Theorem}[section]
\newtheorem{lemma}[theorem]{Lemma}

\newtheorem{proposition}[theorem]{Proposition}
\theoremstyle{definition}
\newtheorem{definition}{Definition}[section]

\theoremstyle{remark}

\newcommand{\dd}{\,\mathrm{d}}
\DeclareMathOperator{\ii}{i}
\DeclareMathOperator{\ee}{e}

\DeclareMathOperator{\artanh}{artanh}
\DeclareMathOperator{\OO}{O}
\DeclareMathOperator{\sinc}{sinc}
\DeclareMathOperator{\Si}{Si}

\DeclareMathOperator{\diag}{diag}

\renewcommand{\Im}{\operatorname{Im}}
\usepackage{tikz}

\begin{document}

\title{Yet another DE-Sinc indefinite integration formula}
\author{Tomoaki Okayama\footnote{Graduate School of Information Sciences,
Hiroshima City University}~~and Ken'ichiro Tanaka\footnote{Graduate School of
Information Science and Technology, The University of Tokyo}}
%\author{Okayama}{Tomoaki}{a}
%\author{Tanaka}{Ken'ichiro}{b}
%\author{Surname3}{Name3}{c}

%\affiliation{a}{Graduate School of Information Sciences, Hiroshima City University}
%\affiliation{b}{Graduate School of Information Science and Technology, The University of Tokyo}
%\affiliation{c}{Affiliation3}

%\volume{7}
%\firstpage{1}

\maketitle

%%%%%%%%%%%%%%%%%%%%%%%%%%%%%%%%%%%%%%%%%%%%%%%%%%%%%%%%%%%%%%%%%%%%

\begin{abstract}
Based on the Sinc approximation
combined with the tanh transformation,
Haber derived an approximation formula
for numerical indefinite integration
over the finite interval (-1, 1).
%The formula, (SE1) uses a special function for the basis functions.
The formula uses a special function for the basis functions.
In contrast,
%Stenger derived another formula (SE2),
Stenger derived another formula,
which does not use any special function but does
include a double sum.
%Subsequently, Muhammad and Mori proposed the formula (DE1),
Subsequently, Muhammad and Mori proposed a formula,
which replaces
the tanh transformation with the double-exponential transformation
in Haber's formula.
%Almost simultaneously, Tanaka et al.\ proposed the formula (DE2),
Almost simultaneously, Tanaka et al.\ proposed another formula,
%which was based on the same replacement in (SE2).
which was based on the same replacement in Stenger's formula.
As they reported, the replacement drastically improves
%the convergence rate of (SE1) and (SE2).
the convergence rate of Haber's and Stenger's formula.
In addition to the formulas above,
%it should be noted that
%Stenger derived yet another indefinite integration formula (SE3)
Stenger derived yet another indefinite integration formula
based on the
Sinc approximation combined with the tanh transformation,
which has an elegant matrix-vector form.
In this paper, we propose the replacement of
the tanh transformation with the double-exponential transformation
%in the formula (SE3).
in Stenger's second formula.
We provide a theoretical analysis as well as a numerical comparison.
\end{abstract}

%%%%%%%%%%%%%%%%%%%%%%%%%%%%%%%%%%%%%%%%%%%%%%%%%%%%%%%%%%%%%%%%%%%%

\section{Introduction}
\label{sec:intro}

Haber~\cite{Haber} derived a numerical indefinite integration formula
of the form
\begin{equation}
 \int_{-1}^{x} f(s) \dd{s}
\approx \sum_{j=-n}^n f(x_j) w_j(x),\quad x\in (-1, 1),
\label{eq:general-form-indef}
\end{equation}
where the weight function $w_j$ depends on $x$,
but the sampling node $x_j$ does not.
Furthermore, the formula can attain a root-exponential convergence:
$\OO(\exp(-c \sqrt{n}))$, even if the integrand $f$
has an endpoint singularity such as $f(s) = 1/\sqrt{1 - s^2}$.
Owing to these features,
the formulas have been used in many applications,
such as initial value problems~\cite{Stenger},
Volterra integral equations~\cite{Rashidinia},
and Volterra integro-differential equations~\cite{Zarebnia}.

Haber's formula is derived as follows.
The Sinc approximation over the real axis is expressed as
\[
 F(u) \approx \sum_{j=-n}^n F(jh) \sinc\left(\frac{u - jh}{h}\right),
\quad u\in\RR,
\]
where $\sinc(x)$ is the normalized sinc function,
and $h$ is the mesh size appropriately selected depending on $n$.
By integrating both sides of the Sinc approximation,
we have
\begin{equation}
 \int_{-\infty}^{t}F(u)\dd u
\approx \sum_{j=-n}^n F(jh) J(j,h)(t),
\quad t\in\RR,
\label{eq:sinc-indef}
\end{equation}
where
\[
 J(j,h)(t)=\int_{-\infty}^{t}\sinc\left(\frac{u - jh}{h}\right)\dd u
=\frac{h}{\pi}\left\{\frac{\pi}{2}+\Si\left(\frac{\pi(t -jh)}{h}\right)\right\}.
\]
Here, $\Si(x)$ is the sine integral defined by
$\Si(x) = \int_0^x \{(\sin t)/t\}\dd t$.
The approximation~\eqref{eq:sinc-indef} is called
the Sinc indefinite integration.
In the case of the indefinite integral over the finite interval $(-1, 1)$,
Haber employed the tanh transformation
\[
 s = \psi(u) = \tanh\left(\frac{u}{2}\right),
\]
and applied~\eqref{eq:sinc-indef} as
\begin{equation}
 \int_{-1}^x f(s)\dd s
= \int_{-\infty}^{\psi^{-1}(x)} f(\psi(u))\psi'(u)\dd u
\approx
\sum_{j=-n}^n f(\psi(jh))\psi'(jh) J(j, h)(\psi^{-1}(x)),\quad x \in (-1, 1).
\label{eq:Haber-A}
\end{equation}
In this paper, we refer to this approximation as formula (SE1).

Although the formula is simple and efficient as described above,
it involves a drawback related to the basis function $J(j,h)$,
which includes a special function $\Si(x)$.
The routine of $\Si(x)$ is not always available,
and has a higher computational cost than that of elementary functions.
In order to remedy this drawback,
introducing an auxiliary function $\eta(x)=(1+x)/2$,
%Haber derived another formula
%\begin{equation}
%\int_{-1}^x f(s)\dd s
%\approx h \sum_{i=-n}^n \left(
%\sum_{j=-n}^n f(\psi(jh))\psi'(jh) \delta^{(-1)}_{ij}\right)
%\sinc\left(\frac{\psi^{-1}(x) - ih}{h}\right)
%+ I^{\ast}_{\psi}
%\left\{\eta(x) - \sum_{i=-n}^n \eta(\psi(ih))\sinc\left(\frac{\psi^{-1}(x)-ih}{h}\right)\right\},
%%\quad x\in (-1, 1),
%\label{eq:Haber-B}
%\end{equation}
%The approximation~\eqref{eq:Haber-B} is referred to as formula (SE2)
%in this paper.
%It should be noted that the formula (SE2) coinsides with (SE1)
%at $x = \psi(kh)$ ($k=-n,\,\ldots,\,n$).
%
%Based on the same idea as the formula (SE2),
Stenger~\cite{Stenger} derived another formula
\begin{equation}
\int_{-1}^x f(s)\dd s
\approx \mathcal{I}_{\SE{}2}(f,x):=
 h \sum_{i=-n}^n \left\{\sum_{j=-n}^n
\left( f(\psi(jh))
-
%\frac{\ee^{jh}}{(1 + \ee^{jh})^2}
\frac{1}{2}
%\frac{1}{4\cosh^2(jh/2)}
% \{(1 - \eta(\psi(jh)))\}\eta(\psi(jh))
 I^{\ast}_{\psi}
\right)\psi'(jh)
\delta^{(-1)}_{ij}\right\}
\sinc\left(\frac{\psi^{-1}(x) - ih}{h}\right)
+ I^{\ast}_{\psi} \eta(x),
%,\quad x\in (-1, 1).
\label{eq:Stenger-B}
\end{equation}
where
\begin{align*}
 \sigma_k &= \int_{0}^{k} \sinc(x)\dd x = \frac{1}{\pi}\Si(\pi k),\\
 I^{\ast}_{\psi} &= h \sum_{k=-n}^n f(\psi(kh))\psi'(kh),
\end{align*}
and $\delta^{(-1)}_{ij} = 1/2 + \sigma_{i-j}$.
We refer to the approximation~\eqref{eq:Stenger-B} as formula (SE2).
%The formula (SE3) can be derived from the formula (SE2) by
%applying the Sinc indefinite integration~\eqref{eq:sinc-indef} as
%\[
% \eta(\psi(ih))
%=\frac{\ee^{ih}}{1 + \ee^{ih}}
%=\int_{-\infty}^{ih} \frac{\ee^{t}}{(1 + \ee^{t})^2}\dd t
%\approx \sum_{j=-n}^n \frac{\ee^{jh}}{(1 + \ee^{jh})^2} J(j,h)(ih)
%%= h \sum_{j=-n}^n \frac{\ee^{jh}}{(1 + \ee^{jh})^2}\delta^{(-1)}_{ij}.
%= h \sum_{j=-n}^n \frac{\psi'(jh)}{2}\delta^{(-1)}_{ij}.
%%= h \sum_{j=-n}^n \{1 - \eta(\psi(jh))\}\eta(\psi(jh)) \delta^{(-1)}_{ij}.
%\]
In view of~\eqref{eq:Stenger-B},
%~\eqref{eq:Haber-B} and,
we need not compute any special function
(note that the value table of $\sigma_k$ is
available~\cite[Table~1.10.1]{Stenger}).
In contrast, however, a double sum exists, which increases the computational cost
compared to single sum.
Therefore, which formula performs better remains unclear.
We note that Haber~\cite{Haber} also derived a similar formula
which uses only elementary functions for the basis functions,
but includes a double sum.
Because those two formulas are essentially the same,
we concentrate on the formula (SE2) in this work.

After a decade, as an improvement of (SE1),
Muhammad and Mori~\cite{MuhammadM} proposed
the replacement of the tanh transformation used in~\eqref{eq:Haber-A}
with the double-exponential (DE) transformation
\[
 s = \phi(u) = \tanh\left(\frac{\pi}{2}\sinh u\right),
\]
and derived the following formula
\begin{equation}
 \int_{-1}^x f(s)\dd s
%= \int_{-\infty}^{\phi^{-1}(x)} f(\phi(u))\phi'(u)\dd u
\approx
\sum_{j=-n}^n f(\phi(jh))\phi'(jh) J(j, h)(\phi^{-1}(x)).
%,\quad x \in (-1, 1).
\label{eq:Muhammad-A}
\end{equation}
We refer to the approximation~\eqref{eq:Muhammad-A} as formula (DE1).
This replacement  has been shown to improve the convergence rate
from $\OO(\exp(-c\sqrt{n}))$
to $\OO(\exp(-c n/\log n))$ in many fields~\cite{Mori,MoriSugi}.
In fact, it was shown~\cite{OkayamaMM} that
the formula (DE1) can attain a convergence rate of $\OO(\exp(-c n/\log n))$.
Such a rapid convergence motivated several authors
to apply the formula to various applications~\cite{MuhammadN,Nurmuhammad,Okayama2}.

Almost simultaneously, as an improvement of (SE2),
Tanaka et al.~\cite{Tanaka} proposed the following formula
\begin{equation}
\int_{-1}^x f(s)\dd s
\approx \mathcal{I}_{\DE{}2}(f,x):=
 h \sum_{i=-n}^n \left\{ \sum_{j=-n}^n
\left( f(\phi(jh))
-
%\frac{\ee^{jh}}{(1 + \ee^{jh})^2}
%\frac{\pi\cosh(jh)}{4\cosh^2((\pi/2)\sinh(jh))}
\frac{1}{2}
% \{(1 - \eta(\psi(jh)))\}\eta(\psi(jh))
 I^{\ast}_{\phi}
\right)\phi'(jh)
\delta^{(-1)}_{ij}\right\}
\sinc\left(\frac{\phi^{-1}(x) - ih}{h}\right)
+ I^{\ast}_{\phi} \eta(x),
%,\quad x\in (-1, 1).
\label{eq:Tanaka-B}
\end{equation}
where
\[
 I^{\ast}_{\phi} = h \sum_{k=-n}^n f(\phi(kh))\phi'(kh).
\]
We refer to the approximation~\eqref{eq:Tanaka-B} as formula (DE2).
They also analyzed the convergence rate as
$\OO(\exp(-c n/\log n))$, which is the same as that of (DE1).

%Akinola made a comparison between (SE1)--(SE3) and (DE3).
%According to the result, the formula (DE3) gave the fastest convergence
%in terms of $n$, but in terms of the CPU time,
%(SE1) is the best.

In addition to the formulas above,
Stenger~\cite{Stenger}
derived yet another indefinite integration formula, expressed as
\begin{equation}
\int_{-1}^x f(s)\dd s
\approx
\sum_{i=-M}^N
\left(h \sum_{j=-M}^N \delta^{(-1)}_{ij}f(\psi(jh))\psi'(jh)\right)
\omega^{\SE}_i(x),
\label{eq:Stenger-C}
\end{equation}
where
\[
 \left\{
\begin{array}{rl}
 \omega^{\SE}_{-M}(x)=&\!\!\!\!\displaystyle
%\frac{2}{1-\psi(-Mh)}
\frac{1}{1 - \eta(\psi(-Mh))}
%(1+\operatorname{e}^{-Nh})
\left\{
%\frac{1-x}{2}
(1 - \eta(x))
 - \sum_{k=-M+1}^N
%\left(\frac{1-\psi(kh)}{2}\right)
(1 - \eta(\psi(kh)))
\sinc\left(\frac{\psi^{-1}(x) - kh}{h}\right)
% \frac{}{1+\operatorname{e}^{kh}}
\right\},\\
 \omega^{\SE}_i(x) =&\!\!\!\!\displaystyle\sinc\left(\frac{\psi^{-1}(x) - ih}{h}\right) \quad\quad
(i=-M+1,\,\ldots,\,N-1),\\
\omega^{\SE}_N(x)=&\!\!\!\!\displaystyle
%\frac{2}{\psi(Nh)+1}
\frac{1}{\eta(\psi(Nh))}
%(1+\operatorname{e}^{-Nh})
\left\{
%\frac{x+1}{2}
\eta(x)
 - \sum_{k=-M}^{N-1}
%\left(\frac{\psi(kh)+1}{2}\right)
\eta(\psi(kh))
\sinc\left(\frac{\psi^{-1}(x) - kh}{h}\right)
% \frac{\operatorname{sinc}[(\psi^{-1}(x) - kh)/h]}{1+\operatorname{e}^{-kh}}
\right\}.
\end{array}
\right.
\]
We refer to the approximation~\eqref{eq:Stenger-C} as formula (SE3).
This formula not only extends the sum from $\sum_{j=-n}^n$ to
$\sum_{j=-M}^N$, but also has an elegant matrix-vector form.
Let $m = M+N+1$,
and let $I^{(-1)}_m$ be a square matrix of order $m$ having
$\delta^{(-1)}_{ij}$ as its $(i, j)$th element.
Moreover, let
\begin{align*}
 D^{\SE}_m
 &= \diag[\psi'(-Mh),\,\psi'(-(M-1)h),\,\ldots,\,\psi'((N-1)h),\,\psi'(Nh)],\\
\boldsymbol{f}^{\SE}_m
 &= [f(\psi(-Mh)),\,f(\psi(-(M-1)h)),\,\ldots,\,f(\psi((N-1)h)),\,f(\psi(Nh))]^{\mathrm{T}},\\
\boldsymbol{\omega}^{\SE}_m(x)
 &=
 [\omega^{\SE}_{-M}(x),\,\omega^{\SE}_{-M+1}(x),\,\ldots,\,\omega^{\SE}_{N-1}(x),\,\omega^{\SE}_{N}(x)],
\end{align*}
and let $A^{\SE}_m = h I^{(-1)}_m D^{\SE}_m$.
%Using the notation above, we
Then, the approximation~\eqref{eq:Stenger-C} can be written as
\[
\int_{-1}^x f(s)\dd s
\approx  \boldsymbol{\omega}^{\SE}_m(x) A^{\SE}_m \boldsymbol{f}^{\SE}_m .
\]
Its convergence rate was analyzed as $\OO(\exp(-c\sqrt{n}))$,
where $n = \max\{M, N\}$.
Of note,
%the formula coinsides with (SE1) at $x = \psi(kh)$ ($k=-n,\,\ldots,\,n$).
this notation may be used for an integral of any order; for example,
\begin{align*}
 f(x) &\approx \boldsymbol{\omega}^{\SE}_m(x) \boldsymbol{f}^{\SE}_m ,\\
\int_{-1}^x f(s)\dd s
&\approx  \boldsymbol{\omega}^{\SE}_m(x) A^{\SE}_m \boldsymbol{f}^{\SE}_m ,\\
\int_{-1}^x \left(\int_{-1}^s f(t) \dd t\right)\dd s
 &\approx \boldsymbol{\omega}^{\SE}_m(x) \left(A^{\SE}_m\right)^2 \boldsymbol{f}^{\SE}_m ,\\
\int_{-1}^x
 \left\{
 \int_{-1}^s
 \left(
   \int_{-1}^t f(u)\dd u
 \right)\dd t
\right\}
\dd s
 &\approx \boldsymbol{\omega}^{\SE}_m(x) \left(A^{\SE}_m\right)^3 \boldsymbol{f}^{\SE}_m ,
\end{align*}
and so forth.
As an application,
this beautiful feature has been used
to derive an approximation formula for indefinite convolution~\cite{StengerC}.
%and their applications.

The main contribution of this paper is an improvement of the formula (SE3).
Based on the same idea as (DE1) and (DE2),
we propose the replacement of the tanh transformation
used in (SE3) with the DE transformation.
Namely, using the notation
\[
 \left\{
\begin{array}{rl}
 \omega^{\DE}_{-M}(x)=&\!\!\!\!\displaystyle
%\frac{2}{1-\phi(-Mh)}
\frac{1}{1 - \eta(\phi(-Mh))}
%(1+\operatorname{e}^{-Nh})
\left\{
%\frac{1-x}{2}
(1 - \eta(x))
 - \sum_{k=-M+1}^N
%\left(\frac{1-\phi(kh)}{2}\right)
(1 - \eta(\phi(kh)))
\sinc\left(\frac{\phi^{-1}(x) - kh}{h}\right)
% \frac{}{1+\operatorname{e}^{kh}}
\right\},\\
 \omega^{\DE}_j(x) =&\!\!\!\!\displaystyle\sinc\left(\frac{\phi^{-1}(x) - jh}{h}\right) \quad\quad
(j=-M+1,\,\ldots,\,N-1),\\
\omega^{\DE}_N(x)=&\!\!\!\!\displaystyle
%\frac{2}{\phi(Nh)+1}
\frac{1}{\eta(\phi(Nh))}
%(1+\operatorname{e}^{-Nh})
\left\{
%\frac{x+1}{2}
\eta(x)
 - \sum_{k=-M}^{N-1}
%\left(\frac{\phi(kh)+1}{2}\right)
\eta(\phi(kh))
\sinc\left(\frac{\phi^{-1}(x) - kh}{h}\right)
% \frac{\operatorname{sinc}[(\phi^{-1}(x) - kh)/h]}{1+\operatorname{e}^{-kh}}
\right\},
\end{array}
\right.
\]
we propose the following formula:
\begin{equation}
\int_{-1}^x f(s)\dd s
\approx \boldsymbol{\omega}^{\DE}_m(x) A^{\DE}_m \boldsymbol{f}^{\DE}_m,
\label{eq:Okayama-C}
\end{equation}
where
\begin{align*}
 D^{\DE}_m
 &= \diag[\phi'(-Mh),\,\phi'(-(M-1)h),\,\ldots,\,\phi'((N-1)h),\,\phi'(Nh)],\\
\boldsymbol{f}^{\DE}_m
 &= [f(\phi(-Mh)),\,f(\phi(-(M-1)h)),\,\ldots,\,f(\phi((N-1)h)),\,f(\phi(Nh))]^{\mathrm{T}},\\
\boldsymbol{\omega}^{\DE}_m(x)
 &=
 [\omega^{\DE}_{-M}(x),\,\omega^{\DE}_{-M+1}(x),\,\ldots,\,\omega^{\DE}_{N-1}(x),\,\omega^{\DE}_{N}(x)],
\end{align*}
and $A^{\DE}_m = h I^{(-1)}_m D^{\DE}_m$.
We refer to the approximation~\eqref{eq:Okayama-C} as formula (DE3).
Furthermore, we show theoretically that the formula (DE3) can attain
the same convergence rate as that of (DE1) and (DE2),
i.e., $\OO(\exp(-c n /\log n))$.

As a subsidiary contribution,
we compared the six formulas
(SE1), (SE2), (SE3),
(DE1), (DE2) and (DE3).
%(the formula (SE2) is removed from the list because theoretical assumptions
%are a bit different from others, and it is essentially the same as
%the formula (SE3)).
Akinola~\cite{Akinola} performed a numerical comparison
between (SE1), (SE2)\footnote{In fact, he also considered
Haber's second formula, which gave results similar to those of the formula (SE2).}
and (DE2),
and reported that the formula (SE1) exhibited the lowest CPU time among
those formulas.
However, he compared CPU times with respect to $n$,
which is not suitable for a comparison in terms of efficiency.
In this work, we compared CPU times with respect to accuracy.
%From the results, we observe that the formula (DE2) is the most efficient.

The remainder of this paper is organized as follows.
In Section~\ref{sec:theo_result},
we present existing and new theoretical results.
In Section~\ref{sec:numer_result},
we show numerical results.
In Section~\ref{sec:proofs},
we provide the proof of a new theorem stated in
Section~\ref{sec:theo_result}.
Section~\ref{sec:conclusion} concludes the work.
Appendix~\ref{sec:rigorous-proof} gives a rigorous error bound used in Section~\ref{sec:proofs}.

\section{Existing and new theorems}
\label{sec:theo_result}

In this section, we present convergence theorems
for the formulas (SE1), (SE2), (SE3), (DE1), (DE2), and (DE3).
In those theorems,
the integrand translated by the tanh transformation
or the DE transformation
($f(\psi(\cdot))$ or $f(\phi(\cdot))$)
is assumed to be analytic on a strip complex domain
$\domD_d = \{\zeta\in\CC : |\Im\zeta| < d\}$
for a positive constant $d$.
That is, $f$ should be analytic on the domain
$\psi(\domD_d)=\{z = \psi(\zeta) : \zeta\in\domD_d\}$
or on
$\phi(\domD_d)=\{z = \phi(\zeta) : \zeta\in\domD_d\}$.

\subsection{Convergence theorems for (SE1), (SE2), and (SE3)}

A convergence theorem of the formula (SE1) was given as follows.
In the theorem,
the sum is extended to $\sum_{j=-M}^N$
instead of $\sum_{j=-n}^n$ in the original formula~\eqref{eq:Haber-A}.

\begin{theorem}[Okayama et al.~{\cite[Theorem~2.9]{OkayamaMM}}]
\label{thm:SE1}
Assume that $f$ is analytic in $\psi(\domD_d)$ for $d$ with $0<d<\pi$,
and that there exist positive constants $K$, $\alpha$ and $\beta$
such that
\begin{equation}
 |f(z)|\leq K |z + 1|^{\alpha-1} |z - 1|^{\beta-1}
\label{eq:f-bound-alpha-beta}
\end{equation}
holds for all $z\in\psi(\domD_d)$.
Let $\mu=\min\{\alpha,\beta\}$,
let $n$ be a positive integer,
and let $h$ be selected by
\begin{equation}
 h = \sqrt{\frac{\pi d}{\mu n}}.
\label{eq:h-SE}
\end{equation}
Furthermore, let $M$ and $N$ be positive integers defined by
\begin{equation}
 \begin{cases}
  M = n,\quad N = \left\lceil\dfrac{\alpha}{\beta}n\right\rceil
& (\text{if}\,\,\, \mu = \alpha)  \vspace*{3pt}\\
  N = n,\quad M = \left\lceil\dfrac{\beta}{\alpha}n\right\rceil
& (\text{if}\,\,\, \mu = \beta)
 \end{cases}
\label{eq:M-N-SE}
\end{equation}
respectively.
Then, there exists a constant $C$ independent of $n$ such that
\[
 \sup_{x\in (-1,1)}
\left|
\int_{-1}^x f(s)\dd s
- \sum_{j=-M}^N f(\psi(jh))\psi'(jh) J(j,h)(\psi^{-1}(t))
\right|
\leq C \exp\left(-\sqrt{\pi d \mu n}\right).
\]
\end{theorem}

A convergence theorem of the formula (SE2) was given as follows.

\begin{theorem}[Stenger~{\cite[Corollary~4.5.3]{Stenger}}]
\label{thm:SE2}
Assume that $f$ is analytic in $\psi(\domD_d)$ for $d$ with $0<d<\pi$,
and that there exist positive constants $K$ and $\mu$ with $\mu\leq 1$
such that
\begin{equation}
 |f(z)|\leq K |z + 1|^{\mu-1} |z - 1|^{\mu-1}
\label{eq:f-bound-mu-mu}
\end{equation}
holds for all $z\in\psi(\domD_d)$.
%Let $\mu=\min\{\alpha,\beta\}$,
Let $n$ be a positive integer,
and let $h$ be selected by~\eqref{eq:h-SE}.
%Furthermore, let $M$ and $N$ be positive integers defined
%by~\eqref{eq:M-N-SE}.
Then, there exists a constant $C$ independent of $n$ such that
\[
 \sup_{x\in (-1,1)}
\left|
\int_{-1}^x f(s)\dd s
- \mathcal{I}_{\SE{}2}(f,x)
%- h \sum_{i=-n}^n \left\{ \sum_{j=-n}^n
%\left( f(\psi(jh))
%-
%%\frac{\ee^{jh}}{(1 + \ee^{jh})^2}
%\frac{1}{2}
%%\frac{1}{4\cosh^2(jh/2)}
%% \{(1 - \eta(\psi(jh)))\}\eta(\psi(jh))
% I^{\ast}_{\psi}
%\right)\psi'(jh)
%\delta^{(-1)}_{ij}\right\}
%\sinc\left(\frac{\psi^{-1}(x) - ih}{h}\right)
%- I^{\ast}_{\psi} \eta(x)
\right|
\leq C \sqrt{n} \exp\left(-\sqrt{\pi d \mu n}\right).
\]
\end{theorem}

A convergence theorem of the formula (SE3) was given as follows.

\begin{theorem}[Stenger~{\cite[Theorem~4.9]{StengerC}}]
\label{thm:SE3}
Assume that $f$ is analytic in $\psi(\domD_d)$ for $d$ with $0<d<\pi$,
and that there exist positive constants $K$, $\alpha$ with $\alpha\leq 1$
and $\beta$ with $\beta\leq 1$
such that~\eqref{eq:f-bound-alpha-beta}
holds for all $z\in\psi(\domD_d)$.
Let $\mu=\min\{\alpha,\beta\}$,
let $n$ be a positive integer,
and let $h$ be selected by~\eqref{eq:h-SE}.
Furthermore, let $M$ and $N$ be positive integers defined
by~\eqref{eq:M-N-SE}.
Then, there exists a constant $C$ independent of $n$ such that
\[
 \sup_{x\in (-1,1)}
\left|
\int_{-1}^x f(s)\dd s
- \boldsymbol{\omega}^{\SE}_m(x) A^{\SE}_m \boldsymbol{f}^{\SE}_m
\right|
\leq C \sqrt{n} \exp\left(-\sqrt{\pi d \mu n}\right).
\]
\end{theorem}

\subsection{Convergence theorems for (DE1), (DE2), and (DE3)}

A convergence theorem of the formula (DE1) was given as follows.
In the theorem,
the sum is extended to $\sum_{j=-M}^N$
instead of $\sum_{j=-n}^n$ in the original formula~\eqref{eq:Muhammad-A}.

\begin{theorem}[Okayama et al.~{\cite[Theorem~2.16]{OkayamaMM}}]
\label{thm:DE1}
Assume that $f$ is analytic in $\phi(\domD_d)$ for $d$ with $0<d<\pi/2$,
and that there exist positive constants $K$, $\alpha$ and $\beta$
such that~\eqref{eq:f-bound-alpha-beta}
holds for all $z\in\phi(\domD_d)$.
Let $\mu=\min\{\alpha,\beta\}$,
let $n$ be a positive integer,
and let $h$ be selected by
\begin{equation}
 h = \frac{\log(2 d n/\mu)}{n}.
\label{eq:h-DE}
\end{equation}
Furthermore, let $M$ and $N$ be positive integers defined by
\begin{equation}
 \begin{cases}
  M = n,\quad N = n - \left\lfloor\dfrac{1}{h}\log\left(\dfrac{\beta}{\alpha}\right)\right\rfloor
& (\text{if}\,\,\, \mu = \alpha) \vspace*{3pt}\\
  N = n,\quad M = n - \left\lfloor\dfrac{1}{h}\log\left(\dfrac{\alpha}{\beta}\right)\right\rfloor
& (\text{if}\,\,\, \mu = \beta)
 \end{cases}
\label{eq:M-N-DE}
\end{equation}
respectively.
Then, there exists a constant $C$ independent of $n$ such that
\[
 \sup_{x\in (-1,1)}
\left|
\int_{-1}^x f(s)\dd s
- \sum_{j=-M}^N f(\phi(jh))\phi'(jh) J(j,h)(\phi^{-1}(t))
\right|
\leq C \frac{\log(2 d n/\mu)}{n}
 \exp\left(-\frac{\pi d n}{\log(2 d n/\mu)}\right).
\]
\end{theorem}

For the formula (DE2),
we can deduce the following convergence theorem
based on the existing result~\cite[Theorem~3.1~with $B=\pi/2$ and $C=1$]{Tanaka}.

\begin{theorem}
\label{thm:DE2}
Assume that $f$ is analytic in $\phi(\domD_d)$ for $d$ with $0<d<\pi/2$,
and that there exist positive constants $K$ and $\mu$ with $\mu\leq 1$
such that~\eqref{eq:f-bound-mu-mu}
holds for all $z\in\phi(\domD_d)$.
%Let $\mu=\min\{\alpha,\beta\}$,
Let $n$ be a positive integer,
and let $h$ be selected by~\eqref{eq:h-DE}.
%Furthermore, let $M$ and $N$ be positive integers defined
%by~\eqref{eq:M-N-SE}.
Then, there exists a constant $C$ independent of $n$ such that
\[
 \sup_{x\in (-1,1)}
\left|
\int_{-1}^x f(s)\dd s
- \mathcal{I}_{\DE{}2}(f,x)
%- h \sum_{i=-n}^n \left\{ \sum_{j=-n}^n
%\left( f(\phi(jh))
%-
%%\frac{\ee^{jh}}{(1 + \ee^{jh})^2}
%\frac{1}{2}
%%\frac{1}{4\cosh^2(jh/2)}
%% \{(1 - \eta(\psi(jh)))\}\eta(\psi(jh))
% I^{\ast}_{\phi}
%\right)\phi'(jh)
%\delta^{(-1)}_{ij}\right\}
%\sinc\left(\frac{\phi^{-1}(x) - ih}{h}\right)
%- I^{\ast}_{\phi} \eta(x)
\right|
\leq C \exp\left(-\frac{\pi d n}{\log(2 d n/\mu)}\right).
\]
\end{theorem}

This paper shows the following convergence theorem
for the formula (DE3); we provide a proof in Section~\ref{sec:proofs}.

\begin{theorem}
\label{thm:DE3}
Assume that $f$ is analytic in $\phi(\domD_d)$ for $d$ with $0<d<\pi/2$,
and that there exist positive constants $K$, $\alpha$ with $\alpha\leq 1$
and $\beta$ with $\beta\leq 1$
such that~\eqref{eq:f-bound-alpha-beta}
holds for all $z\in\phi(\domD_d)$.
Let $\mu=\min\{\alpha,\beta\}$,
let $n$ be a positive integer,
and let $h$ be selected by~\eqref{eq:h-DE}.
Furthermore, let $M$ and $N$ be positive integers defined
by~\eqref{eq:M-N-DE}.
Then, there exists a constant $C$ independent of $n$ such that
\[
 \sup_{x\in (-1,1)}
\left|
\int_{-1}^x f(s)\dd s
- \boldsymbol{\omega}^{\DE}_m(x) A^{\DE}_m \boldsymbol{f}^{\DE}_m
\right|
\leq C \exp\left(-\frac{\pi d n}{\log(2 d n/\mu)}\right).
\]
\end{theorem}

\subsection{Discussion on the convergence rate}

Roughly speaking,
the convergence rates of
the formulas (DE1)--(DE3)
(described in Theorems~\ref{thm:DE1}--\ref{thm:DE3})
are much higher
than those of
the formulas (SE1)--(SE3)
(described in Theorems~\ref{thm:SE1}--\ref{thm:SE3}).
The considerable difference originates from the variable transformations
employed;
the tanh transformation is employed
in the formulas (SE1)--(SE3),
whereas the DE transformation is employed
in the formulas (DE1)--(DE3).
In view of those theorems more precisely,
the convergence rate of the formula (SE1) is slightly better than
that of the formulas (SE2) and (SE3),
and the convergence rate of the formula (DE1) is slightly better than
that of the formulas (DE2) and (DE3).
However, the difference is relatively slight,
which may be observed from the results of the numerical examples
provided in the next section.

\section{Numerical examples}
\label{sec:numer_result}

In this section, we present a numerical comparison
between the formulas (SE1), (SE2), (SE3), (DE1), (DE2), and (DE3).
The computation was performed on MacBook Pro
with 2.4 GHz quad-core CPU (Intel Core i5) with 16GB memory, running
Mac OS X 10.15.7.
The computation programs were implemented in the C programming language with double-precision
floating-point arithmetic, and compiled with Apple Clang version 12.0.0
with the `\texttt{-O2}' option. We used the GNU Scientific Library to compute
$\Si(x)$ for (SE1) and (DE1).
We used CBLAS for matrix-vector multiplication
for (SE2), (SE3), (DE2), and (DE3).
Source code for all programs is available at
\url{https://github.com/okayamat/sinc-indef}.

We consider the following four examples:
\begin{align}
\label{eq:example1}
 \int_{-1}^x \frac{1}{\pi\sqrt{1 - s^2}}\dd s
&=\frac{1}{\pi}\left(\arcsin x + \frac{\pi}{2}\right),\\
\label{eq:example2}
 \int_{-1}^x \frac{1}{4\log 2}\log\left(\frac{1+s}{1-s}\right)\dd s
&=\frac{1}{4\log 2}\left\{
\log(1+x)^{1+x} + \log(1-x)^{1-x} - 2\log 2
\right\},
\\
\label{eq:example3}
 \int_{-1}^x \frac{2}{\pi(1 + s^2)}\dd s
&=\frac{1}{2} + \frac{2}{\pi}\arctan x,\\
\label{eq:example4}
 \int_{-1}^x
\frac{-2[s \{\cos(4\artanh s) + \cosh(\pi)\} + \sin(4\artanh s)]}{\sqrt{\cos(4\artanh s) + \cosh(\pi)}}
\dd s
&=(1-x^2)\sqrt{\cos(4\artanh x) + \cosh(\pi)}.
\end{align}

The integrand in~\eqref{eq:example1}
satisfies the assumptions of
Theorems~\ref{thm:SE1}, \ref{thm:SE2}, and \ref{thm:SE3}
with $\alpha=\beta=1/2$
and $d = \pi_{-}$, where $\pi_{-}$ denotes a slightly smaller number
than $\pi$ (note that $d < \pi$).
We chose $\pi_{-} = 3.14$ in the actual computation.
It also satisfies the assumptions of
Theorems~\ref{thm:DE1}, \ref{thm:DE2}, and \ref{thm:DE3}
with $\alpha=\beta=1/2$ and $d = \pi_{-}/2$ (note that $d < \pi/2$).
The integrand in~\eqref{eq:example2}
satisfies the assumptions of
Theorems~\ref{thm:SE1}, \ref{thm:SE2}, and \ref{thm:SE3}
with $\alpha=\beta=1_{-}$ and $d = \pi_{-}$,
where $1_{-}$ denotes a slightly smaller number than $1$.
We chose $1_{-} = 0.99$ in actual computation.
It also satisfies the assumptions of
Theorems~\ref{thm:DE1}, \ref{thm:DE2}, and \ref{thm:DE3}
with $\alpha=\beta=1_{-}$ and $d = \pi_{-}/2$.
The integrand in~\eqref{eq:example3}
satisfies the assumptions of
Theorems~\ref{thm:SE1}, \ref{thm:SE2}, and \ref{thm:SE3}
with $\alpha=\beta=1$ and $d = \pi_{-}/2$.
It also satisfies the assumptions of
Theorems~\ref{thm:DE1}, \ref{thm:DE2}, and \ref{thm:DE3}
with $\alpha=\beta=1$ and $d = \pi_{-}/6$.

In the case of~\eqref{eq:example4},
the situation is different from the examples above.
The integrand in~\eqref{eq:example4}
satisfies the assumptions of
Theorems~\ref{thm:SE1}, \ref{thm:SE2}, and \ref{thm:SE3}
with $\alpha=\beta=1$ and $d = \pi_{-}/2$,
but it does not satisfy the assumptions of
Theorems~\ref{thm:DE1}, \ref{thm:DE2}, and \ref{thm:DE3}.
According to Tanaka et al.~\cite{TanOkaMatSug},
even in this case, the formulas with the DE transformation
can attain the similar convergence rate to
the formulas with the tanh transformation,
with the same $\alpha$ and $\beta$ and $d = \pi_{-}/6$.
Therefore, we used these values in the computations
of the formulas (DE1), (DE2), and (DE3) in this case.
%Note that, however,
%this does not mean that
%we can use
%Theorems~\ref{thm:DE1}, \ref{thm:DE2}, and \ref{thm:DE3} in this case.

We investigated the maximum errors among
1000 equally-spaced points over the interval $(-1, 1)$.
The results are shown in
Figures~\ref{Fig:Ex1error}--\ref{Fig:Ex4time}.
From Figures~\ref{Fig:Ex1error},
\ref{Fig:Ex2error},
\ref{Fig:Ex3error},
and~\ref{Fig:Ex4error},
it may be observed that
the formulas (SE1)--(SE3) exhibited almost the same convergence profile
with respect to $n$.
The same holds true for the formulas (DE1)--(DE3),
and they were more accurate than the formulas (SE1)--(SE3).
In Figure~\ref{Fig:Ex4error}, the rate of the formulas (DE1)--(DE3)
worsened.
This occurred because the assumptions of
Theorems~\ref{thm:DE1}, \ref{thm:DE2}, and \ref{thm:DE3} are
not satisfied in the case of~\eqref{eq:example4}.
However, they still exhibited similar convergence profiles to
the formulas (SE1)--(SE3) even in this case.

Figures~\ref{Fig:Ex1time},
\ref{Fig:Ex2time},
\ref{Fig:Ex3time},
and~\ref{Fig:Ex4time} show the maximum errors with respect to
the computation time.
From Figures~\ref{Fig:Ex1time},
\ref{Fig:Ex2time},
and~\ref{Fig:Ex3time},
it may be observed that the formula (DE2) was the most efficient,
and the formula (DE3) was the second-most efficient.
Although the same cannot be said for Figure~\ref{Fig:Ex4time},
the two formulas still exhibited a similar convergence profile to
the formulas (SE2) and (SE3).
The formulas (SE1) and (DE1) were not so efficient for all examples.
This may be attributed to the special function $\Si(x)$ included
in the basis functions of the two formulas.

\begin{figure}
\begin{center}
 \begin{minipage}{0.45\linewidth}
  \includegraphics[width=\linewidth]{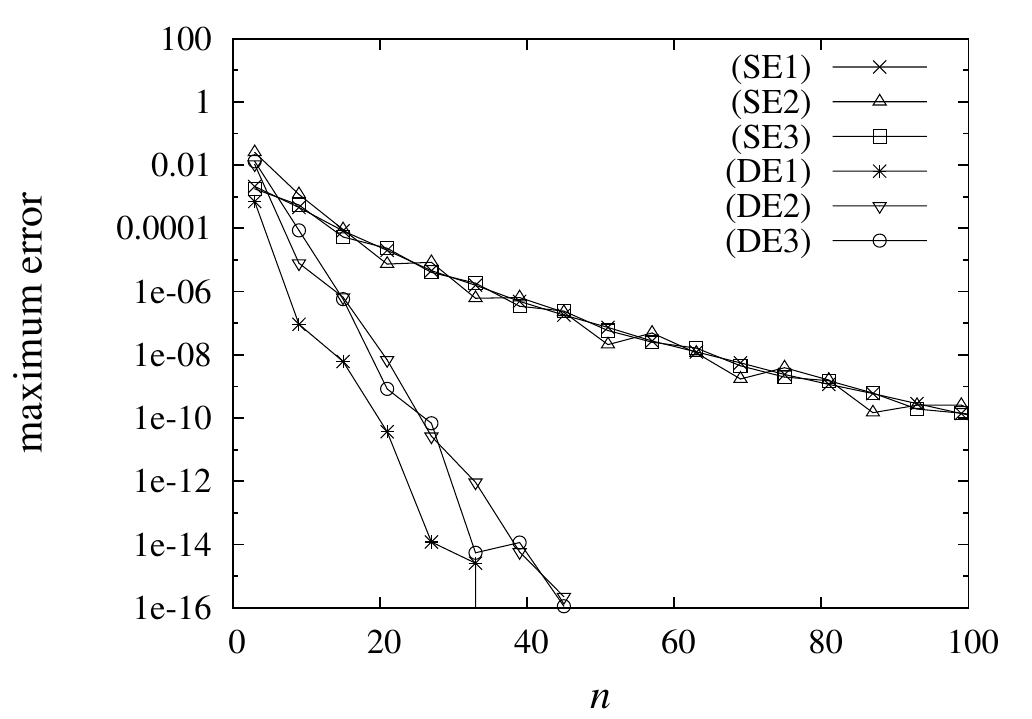}
  \caption{Convergence with respect to $n$ for~\eqref{eq:example1}.}
  \label{Fig:Ex1error}
 \end{minipage}
 \begin{minipage}{0.01\linewidth}
 \mbox{ }
 \end{minipage}
 \begin{minipage}{0.45\linewidth}
  \includegraphics[width=\linewidth]{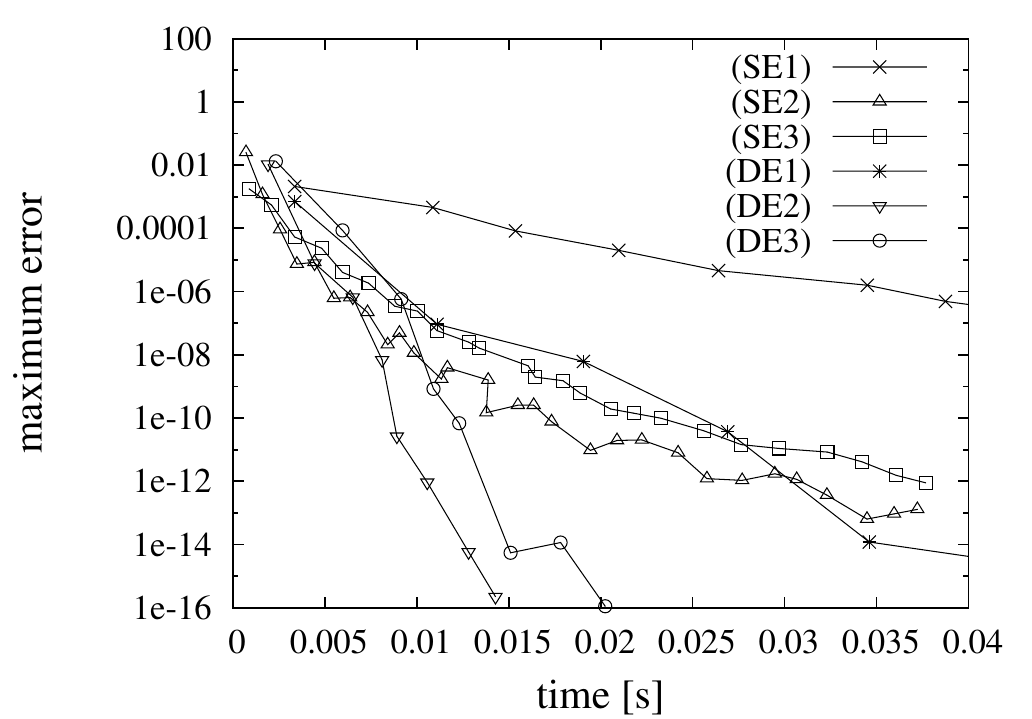}
  \caption{Convergence with respect to computation time for~\eqref{eq:example1}.}
  \label{Fig:Ex1time}
 \end{minipage}
\end{center}
\begin{center}
 \begin{minipage}{0.45\linewidth}
  \includegraphics[width=\linewidth]{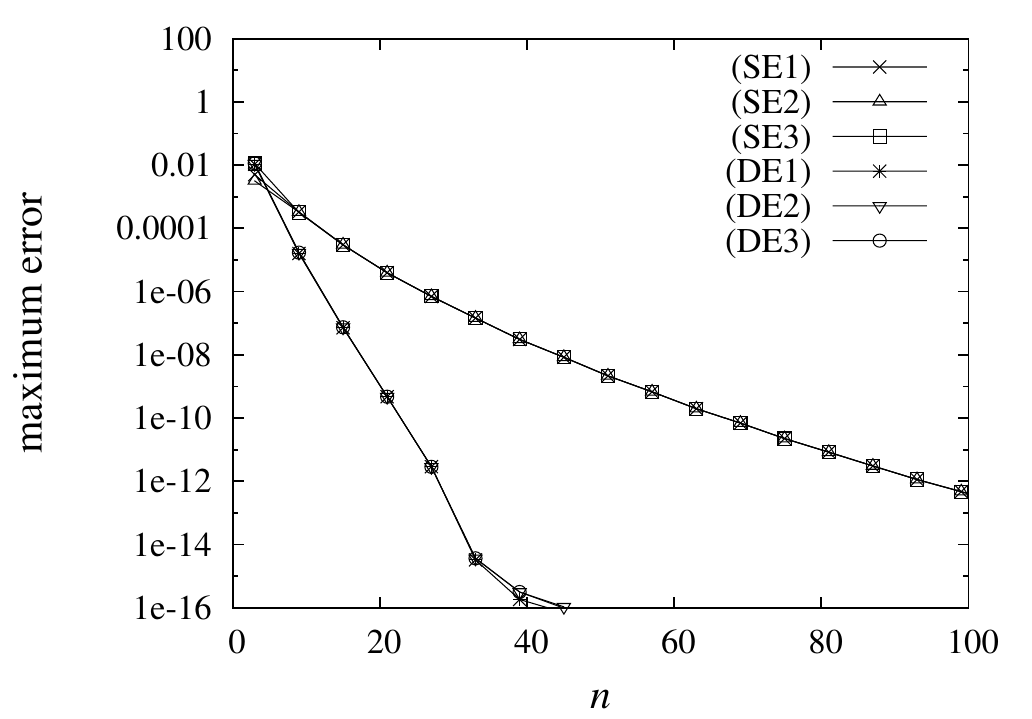}
  \caption{Convergence with respect to $n$ for~\eqref{eq:example2}.}
  \label{Fig:Ex2error}
 \end{minipage}
 \begin{minipage}{0.01\linewidth}
 \mbox{ }
 \end{minipage}
 \begin{minipage}{0.45\linewidth}
  \includegraphics[width=\linewidth]{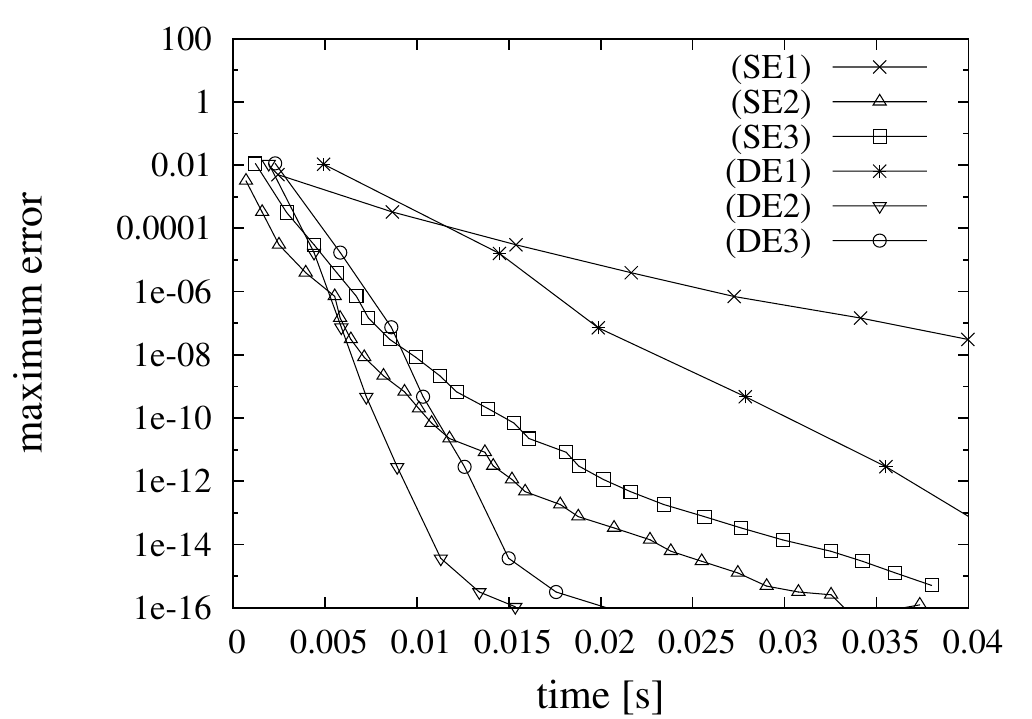}
  \caption{Convergence with respect to computation time for~\eqref{eq:example2}.}
  \label{Fig:Ex2time}
 \end{minipage}
\end{center}
\begin{center}
 \begin{minipage}{0.45\linewidth}
  \includegraphics[width=\linewidth]{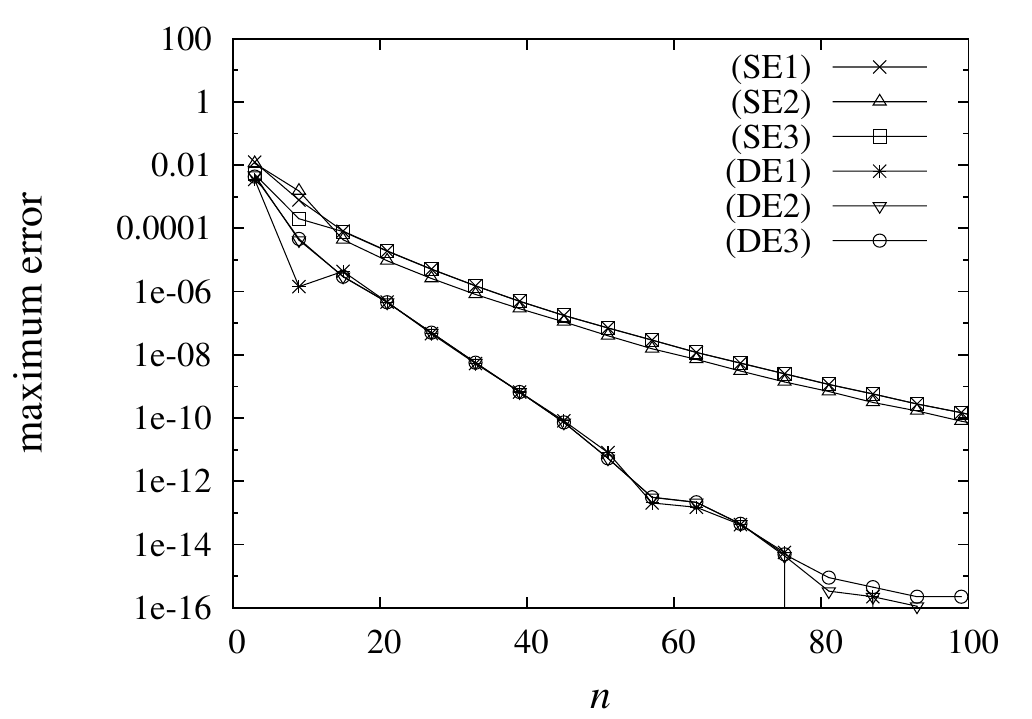}
  \caption{Convergence with respect to $n$ for~\eqref{eq:example3}.}
  \label{Fig:Ex3error}
 \end{minipage}
 \begin{minipage}{0.01\linewidth}
 \mbox{ }
 \end{minipage}
 \begin{minipage}{0.45\linewidth}
  \includegraphics[width=\linewidth]{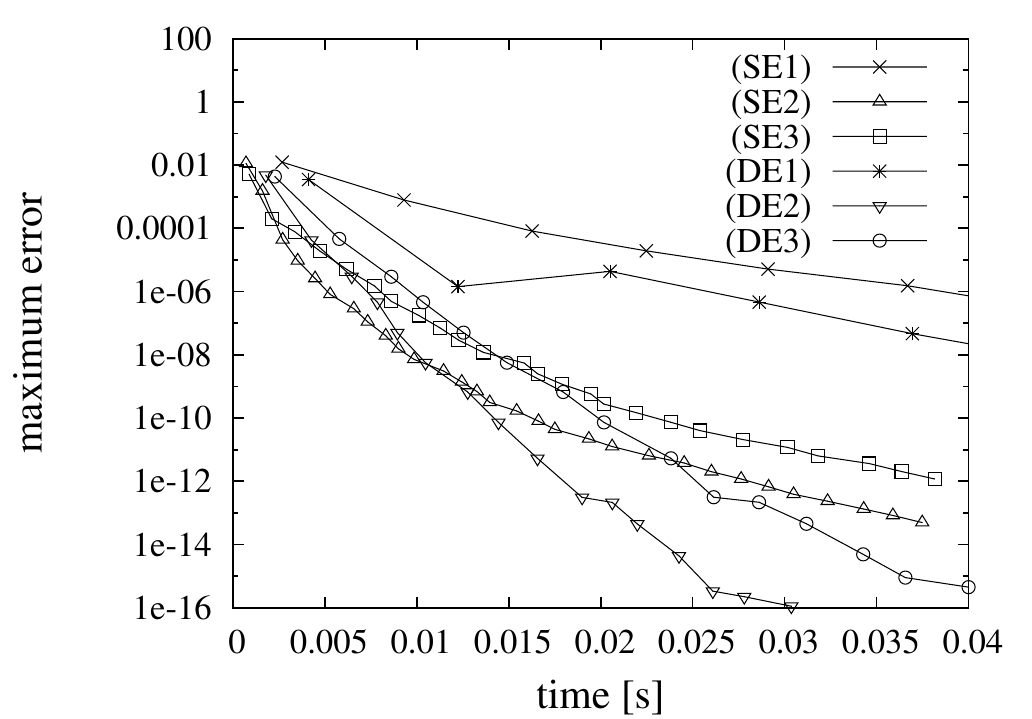}
  \caption{Convergence with respect to computation time for~\eqref{eq:example3}.}
  \label{Fig:Ex3time}
 \end{minipage}
\end{center}
\end{figure}
\begin{figure}
\begin{center}
 \begin{minipage}{0.45\linewidth}
  \includegraphics[width=\linewidth]{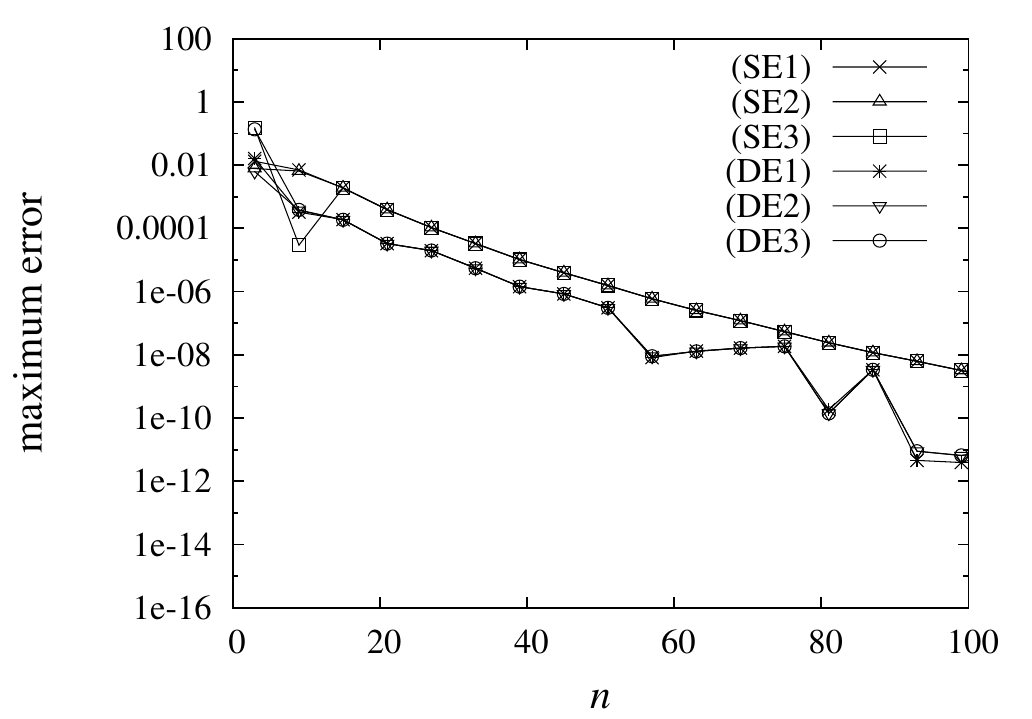}
  \caption{Convergence with respect to $n$ for~\eqref{eq:example4}.}
  \label{Fig:Ex4error}
 \end{minipage}
 \begin{minipage}{0.01\linewidth}
 \mbox{ }
 \end{minipage}
 \begin{minipage}{0.45\linewidth}
  \includegraphics[width=\linewidth]{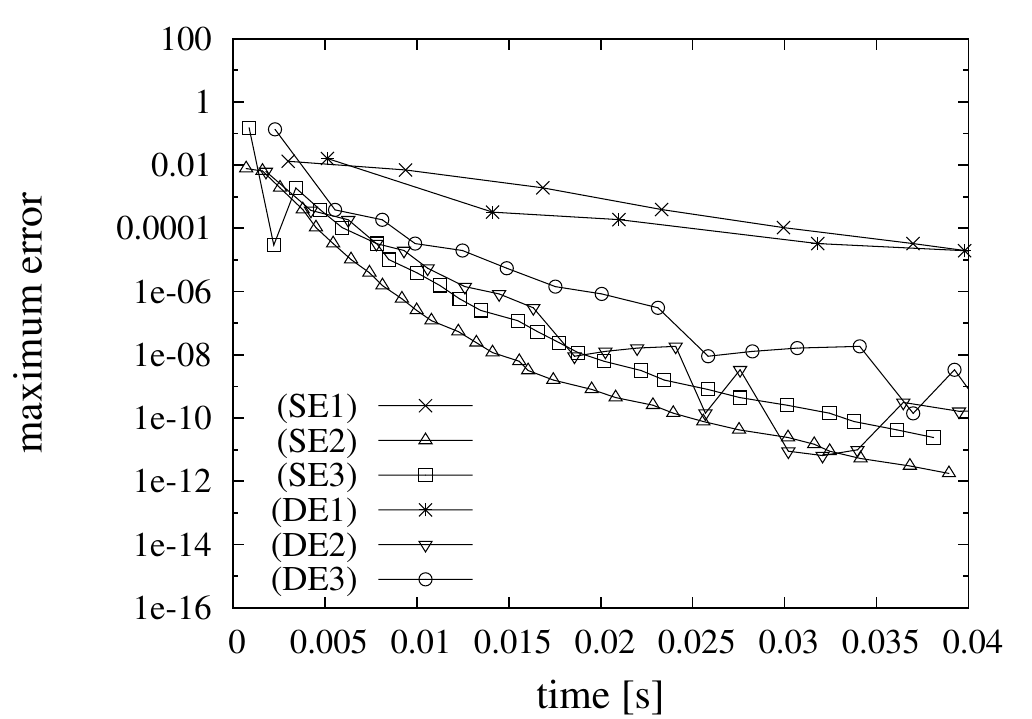}
  \caption{Convergence with respect to computation time for~\eqref{eq:example4}.}
  \label{Fig:Ex4time}
 \end{minipage}
\end{center}
\end{figure}

\section{Proofs}
\label{sec:proofs}

In this section, we provide a proof of Theorem~\ref{thm:DE3}.

%It is organized as follows.
%In Section~\ref{sec:sketch-proof},
%the task is decomposed into three lemmas:
%Lemmas~\ref{lem:bound-discretization-error}
%and~\ref{lem:bound-truncation-error}.
%To prove these lemmas,
%useful inequalities are presented in
%Sections~\ref{subsec:real-ineq}, \ref{subsec:domDdplus}, \ref{subsec:domDdminus},
%and~\ref{subsec:domDd}.
%Following this,
%Lemma~\ref{lem:bound-None} is proved in Section~\ref{subsec:discretization-error
%},
%and Lemma~\ref{lem:bound-truncation-error} is proved in Section~\ref{subsec:trun
%cation-error}.

\subsection{Sketch of the proof}
\label{sec:sketch-proof}

First, we explain the organization of the proof.
Let $F(u)=f(\phi(u))\phi'(u)$ and
\[
 g(x)=\int_{-1}^x f(s)\dd s
 = \int_{-\infty}^{\phi^{-1}(x)} f(\phi(u))\phi'(u)\dd u
 = \int_{-\infty}^{\phi^{-1}(x)} F(u)\dd u.
\]
The main strategy in the proof of Theorem~\ref{thm:DE3} is
to split the error into three terms as follows:
\begin{align*}
\left|\int_{-1}^x f(s)\dd s
 - \boldsymbol{\omega}^{\DE}_m(x) A^{\DE}_m\boldsymbol{f}_m
\right|
&=\left|g(x) - \sum_{i=-M}^N
\left(\sum_{j=-M}^N F(jh) J(j,h)(ih)
\right)\omega^{\DE}_i(x)
\right|\\
&\leq\left|
g(x) - \sum_{i=-M}^N g(\phi(ih))\omega^{\DE}_i(x)
\right|
\\
&\quad
 + \left|
\sum_{i=-M}^N \left(\int_{-\infty}^{ih} F(u)\dd u
- \sum_{j=-\infty}^{\infty}F(jh)J(j,h)(ih)\right)\omega^{\DE}_i(x)
\right|\\
&\quad
+\left|
\sum_{i=-M}^N \left(\sum_{j=-\infty}^{\infty}F(jh)J(j,h)(ih)
-\sum_{j=-M}^{N}F(jh)J(j,h)(ih)
\right)\omega^{\DE}_i(x)
\right|.
%+\left|
%\sum_{i=-M}^N \left(\sum_{j=-\infty}^{-M-1}F(jh)J(j,h)(ih)
% +
% \sum_{j=N+1}^{\infty}F(jh)J(j,h)(ih)\right)\omega^{\DE}_i(x)
%\right|.
\end{align*}
In Section~\ref{sec:bound-first-term}, we bound the first term as
\begin{equation}
\left|
g(x) - \sum_{i=-M}^N g(\phi(ih))\omega^{\DE}_i(x)\right|
\leq
C_1 \exp\left(-\frac{\pi d n}{\log(2 d n/\mu)}\right).
\label{eq:bound-first-term}
\end{equation}
In Section~\ref{sec:bound-second-term}, we bound the second term as
\begin{equation}
 \left|
\sum_{i=-M}^N \left(\int_{-\infty}^{ih} F(u)\dd u
- \sum_{j=-\infty}^{\infty}F(jh)J(j,h)(ih)\right)\omega^{\DE}_i(x)
\right|
\leq
C_2 \log(n+1)
\frac{\log(2 d n/\mu)}{n}
\exp\left(-\frac{\pi d n}{\log(2 d n/\mu)}\right).
\label{eq:bound-second-term}
\end{equation}
In Section~\ref{sec:bound-third-term}, we bound the third term as
\begin{equation}
\left|
\sum_{i=-M}^N \left(\sum_{j=-\infty}^{\infty}F(jh)J(j,h)(ih)
-\sum_{j=-M}^{N}F(jh)J(j,h)(ih)
\right)\omega^{\DE}_i(x)
\right|
% \left|
%\sum_{i=-M}^N \left(\sum_{j=-\infty}^{-M-1}F(jh)J(j,h)(ih)
% +
% \sum_{j=N+1}^{\infty}F(jh)J(j,h)(ih)\right)\omega^{\DE}_i(x)
%\right|
\leq C_3 \log(n+1)
\exp\left(-\pi d n\right).
\label{eq:bound-third-term}
\end{equation}
Here, $C_1$, $C_2$, and $C_3$ are constants, independent of $n$.
This completes the proof of Theorem~\ref{thm:DE3}.

\subsection{Bound of the first term}
\label{sec:bound-first-term}

The following function space is crucial in the analysis.

\begin{definition}
Let $\domD$
be a bounded and simply-connected
domain (or Riemann surface) that contains the interval $(a, b)$.
Let $\alpha$ and $\beta$ be
positive constants with $\alpha\leq 1$ and $\beta\leq 1$.
Then, $\mathbf{M}_{\alpha,\beta}(\domD)$ denotes the family of all functions
$f$ that are analytic and bounded on $\domD$ and satisfy the
following inequalities with a constant $C$
\begin{align*}
|f(z) - f(a)| &\leq C|z - a|^{\alpha},\\
|f(b) - f(z)| &\leq C|b - z|^{\beta}
\end{align*}
for all $z\in\domD$.
\end{definition}

Hereafter, we set $(a, b) = (-1, 1)$.
The next error analysis was given for functions belonging to this function space.

\begin{theorem}[Okayama~{\cite[Theorem 7]{Okayama1}}]
Assume that $f\in\mathbf{M}_{\alpha,\beta}(\phi(\domD_d))$ for $d$
with $0<d<\pi/2$.
Let $\mu = \min\{\alpha, \beta\}$,
let $n$ be a positive integer, and let $h$ be selected
by~\eqref{eq:h-DE}.
Furthermore, let $M$ and $N$ be positive integers defined
by~\eqref{eq:M-N-DE}.
Then, there exists a constant $C$ independent of $n$ such that
\[
 \sup_{x\in (-1,1)}\left|
f(x) - \sum_{i=-M}^N f(\phi(ih))\omega^{\DE}_i(x)
\right|
\leq C \exp\left(-\frac{\pi d n}{\log(2 d n/\mu)}\right).
\]
\end{theorem}

From this theorem,
we obtain~\eqref{eq:bound-first-term} using the following lemma.

\begin{lemma}
Assume that $f$ is analytic in $\phi(\domD_d)$ for $d$ with $0<d<\pi/2$,
and that there exist positive constants $K$, $\alpha$ with $\alpha\leq 1$
and $\beta$ with $\beta\leq 1$
such that~\eqref{eq:f-bound-alpha-beta}
holds for all $z\in\phi(\domD_d)$.
Then, setting $g(x) = \int_{-1}^x f(s)\dd s$, we have
$g\in\mathbf{M}_{\alpha,\beta}(\phi(\domD_d))$.
\end{lemma}

We omit the proof of this lemma because it is a straightforward extension
of an existing result~\cite[Lemma 5]{Okayama3}
(from $\mathbf{M}_{\alpha,\alpha}(\phi(\domD_d))$
to $\mathbf{M}_{\alpha,\beta}(\phi(\domD_d))$).
This completes the proof of~\eqref{eq:bound-first-term}.

\subsection{Bound of the second term}
\label{sec:bound-second-term}

The following function space is crucial in the analysis.

\begin{definition}
\label{def:func_space_BDd}
Let $d$ be a positive real number.
Then, $B(\domD_d)$ denotes the family of all functions
$f$ that are analytic on $\domD_d$
and satisfy
\begin{align}
 \varLambda(f,d)
 := \lim_{c\to d - 0}
\int_{-\infty}^{\infty}
\Big( |f(x+\ii c)|+|f(x-\ii c)| \Big) \dd x
< \infty
\label{eq:real_axis_integral_2_in_mathprel}
\end{align}
and
\begin{align}
\lim_{x \to \pm \infty} \int_{-d}^{d} |f(x + \ii y)| \, \dd y = 0.
\label{eq:imaginary_axis_integral_2_in_mathprel}
\end{align}
\end{definition}

The theorem below holds for functions belonging to this function space;
a proof is provided in Section~\ref{sec:rigorous-proof}.

\begin{theorem}
\label{thm:sinc_int_on_R_disc_err}
Assume that $f\in B(\domD_d)$.
Then, for any $h > 0$, we have
\begin{align}
\label{eq:D_error_of_sinc_int}
\sup_{x\in\mathbb{R}}
\left|\int_{-\infty}^{x}f(t)\dd t
- \sum_{j=-\infty}^{\infty}
f(jh)J(j,h)(x)
\right|
\leq
\frac{4 h \ee^{-\pi d/h}}{\pi d (1 - \ee^{-2\pi d /h})}\varLambda(f,d).
\end{align}
\end{theorem}

We see that $F\in B(\domD_d)$ because the next lemma holds.

\begin{lemma}
Assume that $f$ is analytic in $\phi(\domD_d)$ for $d$ with $0<d<\pi/2$,
and that there exist positive constants $K$, $\alpha$ and $\beta$
such that~\eqref{eq:f-bound-alpha-beta}
holds for all $z\in\phi(\domD_d)$.
Then, putting $F(u)=f(\phi(u))\phi'(u)$, we have
$F\in B(\domD_d)$ and
\[
 \varLambda(F,d)\leq \frac{2^{\alpha+\beta+1} K}{\mu \cos^{\alpha+\beta}((\pi/2)\sin d)\cos d},
\]
where $\mu = \min\{\alpha, \beta\}$.
\end{lemma}

We omit the proof of this lemma as being identical with an existing result~\cite[Lemma 4.6]{OkayamaMM}.
Using this lemma and Theorem~\ref{thm:sinc_int_on_R_disc_err},
we have
\begin{equation}
 \left|
\sum_{i=-M}^N \left(\int_{-\infty}^{ih} F(u)\dd u
- \sum_{j=-\infty}^{\infty}F(jh)J(j,h)(ih)\right)\omega^{\DE}_i(x)
\right|
\leq
\left(\sum_{i=-M}^N |\omega^{\DE}_i(x)|\right)
\cdot \frac{4 h \ee^{-\pi d/h}}{\pi d (1 - \ee^{-2\pi d /h})}\varLambda(F,d).
\label{eq:bound-second-term-2}
\end{equation}
Furthermore, the next lemma holds.

\begin{lemma}
\label{lem:bound-omega-sum}
Let $h>0$, and let $n = \max\{M, N\}$. Then,
there exists a constant $C$ independent of $n$ such that
\[
 \sup_{x\in(-1, 1)}
\sum_{i=-M}^N |\omega^{\DE}_i(x)| \leq C \log(n + 1).
\]
\end{lemma}

This lemma is shown by using the next result.

\begin{lemma}[Stenger~{\cite[p.~142]{Stenger}}]
Let $h>0$. Then, it holds that
\[
 \sup_{t\in\RR}\sum_{i=-n}^n \left|\sinc\left(\frac{t - ih}{h}\right)\right|
\leq \frac{2}{\pi}(3 + \log n).
\]
\end{lemma}

Therefore, the right-hand side of~\eqref{eq:bound-second-term-2}
is further bounded as
\[
 \left(\sum_{i=-M}^N |\omega^{\DE}_i(x)|\right)
\cdot \frac{4 h \ee^{-\pi d/h}}{\pi d (1 - \ee^{-2\pi d /h})}\varLambda(F,d)
\leq C \log(n+1)
\cdot \frac{4 h \ee^{-\pi d/h}}{\pi d (1 - \ee^{-2\pi d /h})}\varLambda(F,d).
\]
Substituting~\eqref{eq:h-DE} to $h$, we obtain~\eqref{eq:bound-second-term}.
%This completes the proof of~\eqref{eq:bound-first-term}.

\subsection{Bound of the third term}
\label{sec:bound-third-term}

Note that the left-hand side of~\eqref{eq:bound-third-term}
is rewritten as
\begin{align*}
&\left|
\sum_{i=-M}^N \left(\sum_{j=-\infty}^{\infty}F(jh)J(j,h)(ih)
-\sum_{j=-M}^{N}F(jh)J(j,h)(ih)
\right)\omega^{\DE}_i(x)
\right|\\
&=
 \left|
\sum_{i=-M}^N \left(\sum_{j=-\infty}^{-M-1}F(jh)J(j,h)(ih)
 +
 \sum_{j=N+1}^{\infty}F(jh)J(j,h)(ih)\right)\omega^{\DE}_i(x)
\right|.
%\label{eq:bound-third-term-2}
\end{align*}

The next lemma is crucial in the analysis.

\begin{lemma}[Okayama et al.~{\cite[Lemma 4.20]{OkayamaMM}}]
Assume that $f$ is analytic in $\phi(\domD_d)$ for $d$ with $0<d<\pi/2$,
and that there exist positive constants $K$, $\alpha$ and $\beta$
such that~\eqref{eq:f-bound-alpha-beta}
holds for all $z\in\phi(\domD_d)$.
Let $\mu = \min\{\alpha, \beta\}$,
let $n$ be a positive integer,
let $M$ and $N$ be positive integers defined
by~\eqref{eq:M-N-DE}
Then, putting $F(u)=f(\phi(u))\phi'(u)$, we have
\[
\sup_{x\in\mathbb{R}}\left|
 \sum_{j=-\infty}^{-M-1}F(jh)J(j,h)(x)
 +
 \sum_{j=N+1}^{\infty}F(jh)J(j,h)(x)
\right|
\leq \tilde{C}
%\frac{1.1 K \ee^{(\pi/2)\nu} 2^{\alpha+\beta} }{\mu}
 \exp\left(-\frac{\pi}{2}\mu \exp(nh)\right),
\]
where $\tilde{C}$ is a positive constant independent of $n$.
\end{lemma}

Using this lemma, we obtain
% can bound the right-hand side
%of~\eqref{eq:bound-third-term-2} as
\[
  \left|
\sum_{i=-M}^N \left(\sum_{j=-\infty}^{-M-1}F(jh)J(j,h)(ih)
 +
 \sum_{j=N+1}^{\infty}F(jh)J(j,h)(ih)\right)\omega^{\DE}_i(x)
\right|
\leq \left(\sum_{i=-M}^N |\omega^{\DE}_i(x)|\right)
\cdot \tilde{C} \exp\left(-\frac{\pi}{2}\mu \exp(nh)\right).
\]
Furthermore, using Lemma~\ref{lem:bound-omega-sum}
and substituting~\eqref{eq:h-DE} to $h$,
we obtain~\eqref{eq:bound-third-term}.

\section{Concluding remarks}
\label{sec:conclusion}

Based on the Sinc indefinite integration, several numerical approximation formulas
of the form~\eqref{eq:general-form-indef} have been presented in the literature.
The formula (SE1) was derived
by combining the Sinc indefinite integration with
the tanh transformation.
To remove a special function $\Si(x)$ from the basis functions,
another formula (SE2) was derived.
In addition, yet another formula (SE3) was derived,
which has an elegant matrix-vector form.
These formulas can attain root-exponential convergence $\OO(\exp(-c\sqrt{n}))$.
After a decade, the formulas (SE1) and (SE2) were improved
as the formulas (DE1) and (DE2).
In these two formulas, the tanh transformation is replaced with
the DE transformation,
and the convergence rate is drastically improved to
$\OO(\exp(-c n /\log n))$.
Motivated by this result,
we improved the formula (SE3)
as formula (DE3) in the present work by the same replacement of the variable transformation.
By theoretical analysis,
we have shown that the proposed formula (DE3)
exhibits the same convergence property as the formulas (DE1) and (DE2).
According to the numerical results,
the formula (DE2) seems the most efficient,
followed by the proposed formula (DE3), in all examples
which satisfy the assumptions of Theorems~\ref{thm:DE1}--\ref{thm:DE3}.
We note that
the formula (DE3) exhibited the second-best performance, and
it inherits the beautiful feature of the formula (SE3).
Furthermore, as pointed out by Stenger~\cite[Eq.~(4.6.10)]{Stenger},
the basis function $\omega^{\SE}_{-M}(x)$ can be replaced with
$\sinc((\psi^{-1}(x) + Mh )/h)$ in the case of indefinite integration
(because indefinite integration approaches $0$ as $x\to -1$).
The same applies to the formula (DE3);
$\omega^{\DE}_{-M}(x)$ can be replaced with
$\sinc((\phi^{-1}(x) + Mh )/h)$.
If the replacement were performed in the numerical implementation, the performance gap
between the formulas (DE2) and (DE3) may be expected to narrow.

We consider the following as a future work.
As described in Section~\ref{sec:intro},
the formula (SE3) has already been applied to
indefinite convolution~\cite{StengerC}
and in further application (see references in~\cite{StengerC}).
We are working to apply the formula (DE3) to those applications,
and expect to report the results soon.

\subsection*{Acknowledgments}

This work was partially supported by the JSPS Grantin-Aid for Young
Scientists (B) Number JP17K14147.

\appendix

\section{Rigorous bound of the discretization error of the Sinc indefinite integration}
\label{sec:rigorous-proof}

In this section,
we consider the discretization error of the Sinc indefinite integration,
namely,
\[
\left|
 \int_{-\infty}^x f(t)\dd t
- \sum_{k=-\infty}^{\infty} f(kh) J(k,h)(x)
\right|,
\]
and its bound.
In fact, there exists a result for the bound~\cite[Lemma 3.6.4]{Stenger}
and it has been used as a standard bound of the error,
but we believe that this has not been rigorously proven.
This is because the change of the order of sum and integration
was used in the proof,
but this is difficult to justify under the given assumptions.
Therefore, we present Theorem~\ref{thm:sinc_int_on_R_disc_err} instead,
and provide a proof here.

%To this end, several preparations are necessary.
%First,
%we introduce a function space of analytic functions on the strip region $\domD_d$.

%\begin{definition}
%\label{def:func_space_BDd}
%Let $d$ be a positive real number. 
%A function space $B(\domD_d)$ is the set of analytic functions on the region $\domD_d$ 
%that satisfy
%\begin{align}
% \varLambda(f,d)
% := \lim_{c\to d - 0}
%\int_{-\infty}^{\infty}
%\Big( |f(x+\ii c)|+|f(x-\ii c)| \Big) \dd x
%< \infty
%\label{eq:real_axis_integral_2_in_mathprel}
%\end{align}
%and
%\begin{align}
%\lim_{x \to \pm \infty} \int_{-d}^{d} |f(x + \ii y)| \, \dd y = 0.
%\label{eq:imaginary_axis_integral_2_in_mathprel}
%\end{align}
%\end{definition}

First, we can derive the simple fact that
functions in $B(\domD_d)$ are integrable on the real line
in the sense of the Riemann integral.

\begin{proposition}
\label{prop:B_D_integrable}
If $f \in B(\domD_d)$, then
\[
 \left|\int_{-\infty}^{\infty} f(x)\dd x \right|
\leq  \frac{1}{2}\varLambda(f,d).
\]
\end{proposition}

\begin{proof}
Let $a$ be a positive real number, 
let $\varepsilon$ be a real number with $0 < \varepsilon < d$, and 
let $d_{\varepsilon} = d - \varepsilon$. 
Furthermore, 
we define a contour $C_{a, \varepsilon}$ in $\domD_{d}$ by the rectangle
\begin{align*}
%\label{eq:C_a_varepsilon}
 \{ z = \pm a + \ii y \mid 0 \leq y \leq d_{\varepsilon} \}
%\notag
\ \cup \ 
\{ z = x \mid -a \leq x \leq a \}
\ \cup \ 
\{ z = x + \ii d_{\varepsilon} \mid -a \leq x \leq a \}
\end{align*}
with counterclockwise direction (see Figure~\ref{fig:cntr_C_a_varepsilon}).
Because $f$ is analytic in $\domD_{d}$, it follows from Cauchy's integral theorem that 
\begin{align}
\notag
\oint_{C_{a, \varepsilon}} f(z) \dd z = 0. 
\end{align}
\begin{figure}[ht]
\begin{center}
\begin{tikzpicture}
\draw[->,>=stealth,semithick] (-4,0)--(4,0)node[above]{$\RR$};
\draw[->,>=stealth,semithick] (0,-0.5)--(0,3)node[above]{$\ii \RR$};
\draw (0,0)node[below left]{$\OO$};
\draw (-3.3,0)node[below]{$-a$};
\draw (3.3,0)node[below]{$a$};
\draw[->,>=stealth,ultra thick] (-3.3,0.5)--(-3.3,0)--(2,0);
\draw[->,>=stealth,ultra thick] (2,0)--(3.3,0)--(3.3,1.5);
\draw[->,>=stealth,ultra thick] (3.3,1.5)--(3.3,2)--(-1.5,2);
\draw[->,>=stealth,ultra thick] (-1.5,2)--(-3.3,2)--(-3.3,0.5);
\draw[thick] (-0.1,2.25)--(0.1,2.25)node[above right]{\small$\ii d$};
%\draw[thick] (3.3,0)--(4,2)node[above]{\small$(n+1/2)h$};
%\draw[thick] (-3.3,0)--(-4,2)node[above]{\small$-(n+1/2)h$};
\draw[thick] (2,2)--(2.5,2.5)node[above right]{\small$C_{a,\varepsilon}$};
\draw[thick] (0,2)--(-1.5,2.5)node[above left]{\small$\ii d_{\varepsilon}$};
%\draw[thick] (0,-2)--(-1.5,-2.5)node[below left]{\small$-\ii d_{\varepsilon}$};
\end{tikzpicture}
\caption{Contour $C_{a,\varepsilon}$}
\label{fig:cntr_C_a_varepsilon}
\end{center}
\end{figure}
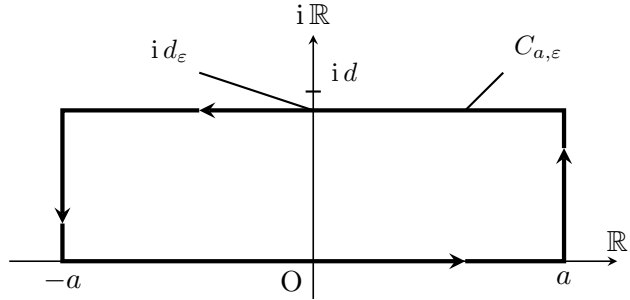
As $a \to \infty$, 
the integrals of $f$ on the vertical parts of $C_{a, \varepsilon}$ converge to $0$
owing to condition \eqref{eq:imaginary_axis_integral_2_in_mathprel}. 
Therefore we have
\begin{align}
\notag
 \int_{-\infty}^{\infty} f(x) \dd x 
 = 
 \int_{-\infty}^{\infty}  f(x + \ii d_{\varepsilon}) \dd x.
\end{align}
Similarly, 
we can derive a variant of this relation in which the sign of $d_{\varepsilon}$ is opposite. 
Therefore, we have
\begin{align*}
%\label{eq:expr_of_int_in_mathprel}
\int_{-\infty}^{\infty} f(x)\dd x
=
 \frac{1}{2} \int_{-\infty}^{\infty} 
 \Big(
 f(x + \ii d_{\varepsilon}) + 
 f(x - \ii d_{\varepsilon})
 \Big) \dd x.
\end{align*}
Taking the absolute values of both sides, we obtain the conclusion. 
\end{proof}

Next,
we present a well-known expression of the error of the Sinc approximation
by a contour integral.

\begin{proposition}[Stenger~{\cite[Eq.~(3.1.6)]{Stenger}}]
Let $n$ be a positive integer, 
let $\varepsilon$ be a real number with $0 < \varepsilon < d$, and 
let $\varGamma_{n, \varepsilon}$ be the contour given by 
\begin{align}
%\label{eq:varGamma}
\notag
 \{ z = \pm (n+1/2)h + \ii y \mid -d_{\varepsilon} \leq y \leq d_{\varepsilon} \}
\notag
\ \cup \ 
\{ z = x \pm \ii d_{\varepsilon} \mid -(n+1/2)h \leq x \leq (n+1/2)h \}
\end{align}
with counterclockwise direction (see Figure \ref{fig:cntr_for_trpzd}),
where $d_{\varepsilon} = d - \varepsilon$.
When $x$ is contained in the region surrounded by $\varGamma_{n, \varepsilon}$ (i.e., $n$ is sufficiently large), 
\begin{align}
\label{eq:sinc-complex-int}
f(x) - \sum_{k=-n}^n f(kh) \sinc\left(\frac{x - kh}{h}\right)
=
\frac{\sin(\pi x/h)}{2\pi \ii}
\oint_{\varGamma_{n, \varepsilon}}
\frac{f(z)}{(z - x)\sin(\pi z/h)}\dd z
\end{align}
holds. 
\end{proposition}

%\begin{figure}[ht]
%\begin{center}
%\includegraphics[scale=.6]{intpathGamma.eps}
%\end{center}
%\caption{Contour $\varGamma_{n,\varepsilon}$}
%\label{fig:cntr_for_trpzd}
%\end{figure}

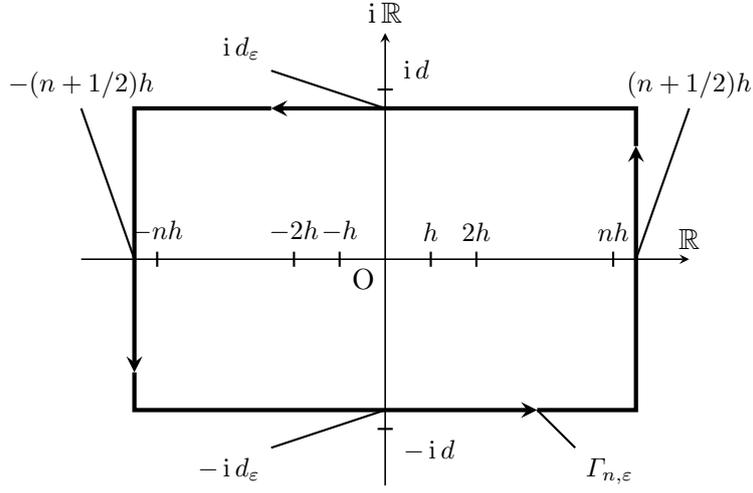
\begin{figure}[ht]
\begin{center}
\begin{tikzpicture}
\draw[->,>=stealth,semithick] (-4,0)--(4,0)node[above]{$\RR$};
\draw[->,>=stealth,semithick] (0,-3)--(0,3)node[above]{$\ii \RR$};
\draw (0,0)node[below left]{$\OO$};
\draw[->,>=stealth,ultra thick] (-3.3,-1.5)--(-3.3,-2)--(2,-2);
\draw[->,>=stealth,ultra thick] (2,-2)--(3.3,-2)--(3.3,1.5);
\draw[->,>=stealth,ultra thick] (3.3,1.5)--(3.3,2)--(-1.5,2);
\draw[->,>=stealth,ultra thick] (-1.5,2)--(-3.3,2)--(-3.3,-1.5);
\draw[thick] (-3,-0.1)--(-3,0.1)node[above]{\small$-nh$};
\draw[thick] (-0.6,-0.1)--(-0.6,0.1)node[above]{\small$-h$};
\draw[thick] (-1.2,-0.1)--(-1.2,0.1)node[above]{\small$-2h$};
\draw[thick] (3,-0.1)--(3,0.1)node[above]{\small$nh$};
\draw[thick] (0.6,-0.1)--(0.6,0.1)node[above]{\small$h$};
\draw[thick] (1.2,-0.1)--(1.2,0.1)node[above]{\small$2h$};
\draw[thick] (-0.1,2.25)--(0.1,2.25)node[above right]{\small$\ii d$};
\draw[thick] (-0.1,-2.25)--(0.1,-2.25)node[below right]{\small$-\ii d$};
\draw[thick] (3.3,0)--(4,2)node[above]{\small$(n+1/2)h$};
\draw[thick] (-3.3,0)--(-4,2)node[above]{\small$-(n+1/2)h$};
\draw[thick] (2,-2)--(2.5,-2.5)node[below right]{\small$\varGamma_{n,\varepsilon}$};
\draw[thick] (0,2)--(-1.5,2.5)node[above left]{\small$\ii d_{\varepsilon}$};
\draw[thick] (0,-2)--(-1.5,-2.5)node[below left]{\small$-\ii d_{\varepsilon}$};
\end{tikzpicture}
\caption{Contour $\varGamma_{n,\varepsilon}$}
\label{fig:cntr_for_trpzd}
\end{center}
\end{figure}

To bound the discretization error of the Sinc indefinite integration, 
we integrate the right-hand side of equality \eqref{eq:sinc-complex-int} and bound it
as shown in the proof of Theorem \ref{thm:sinc_int_on_R_disc_err} below. 
In this bounding, the following lemma is key.

\begin{lemma}
Let 
$x$, $y$, $s$, and $h$ be real numbers
and assume that $h > 0$. 
Then, we have
\begin{align}
\label{eq:bound-w-as-int-sinc-general}
\left|
\frac{1}{2\pi\ii}\int_{-\infty}^{x}\frac{\sin(\pi t/h)}{(s+\ii y) - t}\dd t
\right|
&\leq \frac{3}{4\pi|y|} h\quad (y\neq 0),\\
\label{eq:bound-w-as-int-sinc-special-1}
\left|
\frac{1}{2\pi\ii}\int_{-\infty}^{-(n+1)h}\frac{\sin(\pi t/h)}{(\pm(n+1/2)h+\ii y) - t}\dd t
\right|
&\leq \frac{1}{\pi}\left(2+\frac{5|y|}{h}\right).
\end{align}
Furthermore, 
when $- n h \leq x\leq n h$,  
we have 
\begin{align}
\label{eq:bound-w-as-int-sinc-general-2}
\left|
\frac{1}{2\pi\ii}\int_{-nh}^{x}\frac{\sin(\pi t/h)}{(s+\ii y) - t}\dd t
\right|
&\leq \frac{3}{4\pi|y|} h\quad (y\neq 0),\\
\label{eq:bound-w-as-int-sinc-special-2}
\left|
\frac{1}{2\pi\ii}\int_{-nh}^{x}\frac{\sin(\pi t/h)}{(\pm(n+1/2)h+\ii y) - t}\dd t
\right|
&\leq \frac{1}{\pi}\left(2+\frac{5|y|}{h}\right). 
\end{align}
\end{lemma}

Because this lemma can be proven similarly to existing results, we omit its proof. 
In fact, 
to obtain inequality~\eqref{eq:bound-w-as-int-sinc-general}, 
we need only to tighten inequality~(3.6.24) of \cite[Lemma 3.6.3]{Stenger} slightly. 
In addition, 
inequalities~\eqref{eq:bound-w-as-int-sinc-special-1}--\eqref{eq:bound-w-as-int-sinc-special-2}
are owing to Lemmas~4.11 and~4.12 of \cite{TanOkaMatSug}. 

Based on the preparations above, 
we can show Theorem~\ref{thm:sinc_int_on_R_disc_err} as follows.
% the following theorem giving a rigorous bound of the discretization error. 

%\begin{theorem}
%%\label{thm:sinc_int_on_R_disc_err}
%Assume that 
%$f\in B(\domD_d)$. 
%Then, for any $h > 0$, we have 
%\begin{align}
%%\label{eq:D_error_of_sinc_int}
%\sup_{x\in\mathbb{R}}
%\left|\int_{-\infty}^{x}f(t)\dd t
%- \sum_{k=-\infty}^{\infty}
%f(kh)J(k,h)(x)
%\right|
%\leq
%\frac{4 h \ee^{-\pi d/h}}{\pi d (1 - \ee^{-2\pi d /h})}\varLambda(f,d). 
%\end{align}
%\end{theorem}

\begin{proof}[Proof of Theorem~\ref{thm:sinc_int_on_R_disc_err}]
First,
we consider the truncated sum with $-n \leq k \leq n$
in the left-hand side of inequality~\eqref{eq:D_error_of_sinc_int}. 
Then,
by letting $n \to \infty$,
we obtain the bound.
To this end,
we partition the integral with the truncated sum as
\begin{align*}
\int_{-\infty}^x \left\{f(t)
- \sum_{k=-n}^n f(kh) \sinc\left(\frac{t - kh}{h}\right)
\right\}\dd t
&=\int_{-nh}^{x}
\left\{f(t)
- \sum_{k=-n}^n f(kh) \sinc\left(\frac{t - kh}{h}\right)
\right\}\dd t\\
&\quad+\int_{-(n+1)h}^{-nh}
 \left\{f(t)
- \sum_{k=-n}^n f(kh) \sinc\left(\frac{t - kh}{h}\right)
\right\}\dd t\\
&\quad+\int_{-\infty}^{-(n+1)h}
 \left\{f(t)
- \sum_{k=-n}^n f(kh) \sinc\left(\frac{t - kh}{h}\right)
\right\}\dd t,
\end{align*}
and estimate each term.
It suffices to consider the case in which $n$ is sufficiently large
that $-nh\leq x\leq nh$ holds.

We begin with the first term.
By integration of equality~\eqref{eq:sinc-complex-int},
the first term is rewritten as
\begin{align*}
 \int_{-nh}^{x} f(t)\dd t
 - \sum_{k=-n}^n f(kh) \int_{-nh}^x \sinc\left(\frac{t - kh}{h}\right)\dd t
=
\oint_{\varGamma_{n, \varepsilon}}
\left\{
\int_{-nh}^{x}
\frac{\sin(\pi t/h)}{2\pi \ii(z - t)}
\dd t
\right\}
\frac{f(z)}{\sin(\pi z/h)}\dd z. 
\end{align*}
Based on the conditions~\eqref{eq:real_axis_integral_2_in_mathprel}
and~\eqref{eq:imaginary_axis_integral_2_in_mathprel}
defining $B(\domD_d)$, 
we can bound the integrals on the horizontal and vertical paths of $\varGamma_{n, \varepsilon}$
when $n \to \infty$ and $\varepsilon\to +0$. 
For the integral on the horizontal paths, 
we use inequality~\eqref{eq:bound-w-as-int-sinc-general-2}. 
In addition, 
for the integral on the vertical paths, 
we use inequalities~\eqref{eq:bound-w-as-int-sinc-special-2} and~\eqref{eq:bound-w-as-int-sinc-general-2}
for the cases $|y|\leq h$ and $|y|> h$, respectively. 
As a result, we have 
\begin{align*}
&\lim_{\substack{n\to\infty\\ \varepsilon\to +0}}\left|
\oint_{\varGamma_{n, \varepsilon}}
\left\{\int_{-nh}^{x}
\frac{\sin(\pi t/h)}{2\pi \ii(z - t)}
\dd t\right\}
\frac{f(z)}{\sin(\pi z/h)}\dd z
\right|\\
&\leq\frac{3h}{4 \pi d \sinh(\pi d/h)}
\lim_{n\to\infty}
\int_{-(n+1/2)h}^{(n+1/2)h}
\left(
 |f(t + \ii d)| + |f(t - \ii d)|
\right) \dd t \\
&\quad +
\frac{1}{\pi}\left(2+\frac{5h}{h}\right)\lim_{n\to\infty}
\int_{|y|\leq h} (|f((n+1/2)h + \ii y)| + |f(-(n+1/2)h + \ii y)|)\dd y\\
&\quad +
\frac{3h}{4\pi h}\lim_{n\to\infty}
\int_{h<|y|\leq d} (|f((n+1/2)h + \ii y)| + |f(-(n+1/2)h + \ii y)|)\dd y\\
&\leq \frac{3h}{4 \pi d \sinh(\pi d/h)}
\lim_{n\to\infty}
\int_{-(n+1/2)h}^{(n+1/2)h}
\left(
 |f(t + \ii d)| + |f(t - \ii d)|
\right) \dd t \\
&\quad + \lim_{n\to\infty}
\frac{7}{\pi}\int_{-d}^d (|f((n+1/2)h + \ii y)| + |f(-(n+1/2)h + \ii y)|)\dd y\\
&=\frac{3h\ee^{-\pi d/h}}{2 \pi d(1 - \ee^{-2\pi d/h})}
\varLambda(f,d) + 0. 
\end{align*}

Next, we bound the third term. 
By complex contour integral, we have
\begin{align*}
& \int_{-\infty}^{-(n+1)h} f(t)\dd t
 - \sum_{k=-n}^n f(kh) J(k,h)(-(n+1)h)\\
&=
\int_{-\infty}^{-(n+1)h}f(t)\dd t
-
\oint_{\varGamma_{n, \varepsilon}}
\left\{
\int_{-\infty}^{-(n+1)h}
\frac{\sin(\pi t/h)}{2\pi \ii(z - t)}
\dd t
\right\}
\frac{f(z)}{\sin(\pi z/h)}\dd z.
\end{align*}
%Here, note that $t$ is not contained in $\varGamma_{n, \varepsilon}$.
Here, note that $t$ is not contained in the region surrounded by
$\varGamma_{n, \varepsilon}$.
Because it follows from Proposition~\ref{prop:B_D_integrable} that
\[
 \lim_{n\to\infty}\int_{-\infty}^{-(n+1)h}f(t)\dd t
= 0, 
\]
we need only to bound the part of the contour integral similarly to the first term. 
For the integral on the horizontal paths, 
we use the inequality~\eqref{eq:bound-w-as-int-sinc-general}. 
In addition, 
for the integral on the vertical paths, 
we use inequalities~\eqref{eq:bound-w-as-int-sinc-special-1} and~\eqref{eq:bound-w-as-int-sinc-general}
for the cases $|y|\leq h$ and $|y|> h$, respectively. 
As a result, we have 
\begin{align*}
\lim_{\substack{n\to\infty\\ \varepsilon\to +0}}\left|
\oint_{\varGamma_{n, \varepsilon}}
\left\{\int_{-\infty}^{-(n+1)h}
\frac{\sin(\pi t/h)}{2\pi \ii(z - t)}
\dd t\right\}
\frac{f(z)}{\sin(\pi z/h)}\dd z
\right|
\leq\frac{3h\ee^{-\pi d/h}}{2\pi d(1 - \ee^{-2\pi d/h})}
\varLambda(f,d) + 0.
\end{align*}

Finally, 
we bound the second term. 
We replace the contour $\varGamma_{n,\varepsilon}$ with $\varGamma_{n-1,\varepsilon}$. 
Noting that $-(n+1)h\leq t\leq -nh$,
we have
\begin{align*}
&\int_{-(n+1)h}^{-nh}
 \left\{f(t)
- \sum_{k=-n}^n f(kh) \sinc\left(\frac{t - kh}{h}\right)
\right\}\dd t\\
&=\int_{-(n+1)h}^{-nh}
 \left\{f(t)
- f(-nh)\sinc\left(\frac{t + nh}{h}\right)
 - f(nh)\sinc\left(\frac{t - nh}{h}\right)
\right\}\dd t\\
&\quad -
\int_{-(n+1)h}^{-nh}\left\{
 \frac{\sin(\pi t/h)}{2\pi\ii}
\oint_{\varGamma_{n-1,\varepsilon}}
\frac{f(z)}{(z-t)\sin(\pi z/h)}\dd z
\right\}\dd t. 
\end{align*}
Then, 
by using $|\sinc((t - kh)/h))|\leq 1$, we have
\begin{align*}
&\left|
\int_{-(n+1)h}^{-nh}
 \left\{f(t)
- \sum_{k=-n}^n f(kh) \sinc\left(\frac{t - kh}{h}\right)
\right\}\dd t
\right|\\
&\leq
h\left\{\max_{-(n+1)h\leq t\leq -nh}|f(t)|
+|f(-nh)|+|f(nh)|\right\}\\
&\quad+ h \max_{-(n+1)h\leq t\leq -nh}
\left\{
\frac{|\sin(\pi t/h)|}{|2\pi \ii|}
\left|
 \oint_{\varGamma_{n-1,\varepsilon}}\frac{f(z)}{(z-t)\sin(\pi z/h)}
\dd z
\right|\right\}.
\end{align*}
As $n \to \infty$ and $\varepsilon \to 0$, 
the last term of the right-hand side is bounded as
\[
 h \lim_{\substack{n\to\infty\\ \varepsilon\to +0}}
\frac{|\sin(\pi t/h)|}{|2\pi \ii|}
\left|
 \oint_{\varGamma_{n-1,\varepsilon}}\frac{f(z)}{(z-t)\sin(\pi z/h)}
\dd z
\right|
\leq \frac{h \ee^{-\pi d/h}}{\pi d(1 - \ee^{-2\pi d/h})}\varLambda(f,d). 
\]
The other terms tend to $0$ as $n\to\infty$
because $\lim_{t\to\pm\infty} |f(t)| = 0$, 
which follows from Proposition \ref{prop:B_D_integrable}. 

Combining all the bounds, we obtain inequality~\eqref{eq:D_error_of_sinc_int}. 
\end{proof}

%%%%%%%%%%%%%%%%%%%%%%%%%%%%%%%%%%%%%%%%%%%%%%%%%%%%%%%%%%%%%%%%%%%%%%%%%%%%%%%%%%%

\end{document}